%% file: cyrigid3.tex
\documentclass[12pt]{amsart}
\usepackage{amsmath}
\usepackage{amssymb}
\usepackage{amsfonts}
\usepackage{amsthm}
\usepackage{verbatim}
\usepackage{amscd}
\usepackage{cite}
\usepackage{leftidx}
\usepackage{enumerate}
\usepackage{txfonts}
\usepackage{manfnt}
\usepackage{amscd}
\usepackage[mathscr]{eucal}
\usepackage{hyperref}
\usepackage{datetime2}
\usepackage{bookman}
\def\cP{\mathcal P}
\def\cX{\mathcal X}

\newtheorem{thm}{Theorem} 

\newtheorem*{thm*}{Theorem}
\newtheorem*{prop*}{Proposition}
\newtheorem{cor}[thm]{Corollary}
\newtheorem*{cor*}{Corollary}

\newtheorem{lem}[thm]{Lemma}
\newtheorem*{lem*}{Lemma}

\newtheorem*{claim*}{Claim}
\newtheorem{prop}[thm]{Proposition}

\newtheorem{defn}[thm]{Definition}
\theoremstyle{remark}

\newtheorem{rem}[thm]{Remark}
\newtheorem*{rem*}{Remark}
\newtheorem{crit-rem}[thm]{Critical remark}

\newtheorem{example}[thm]{Example}
\newtheorem*{example*}{Example}

\newtheorem*{defn*}{Definition}
\newtheorem*{con*}{Conjecture}

\input{./myheader3.tex}

\begin{document} 

\title{Regular and rigid curves on some Calabi-Yau and general-type
	complete intersections }
\author 
{Ziv Ran}


\thanks{arxiv.org 2207.01096   }
\date {\DTMnow}


\address {\nl UC Math Dept. \nl
Skye Surge Facility, Aberdeen-Inverness Road
\nl
Riverside CA 92521 US\nl 
ziv.ran @  ucr.edu\nl
\url{https://profiles.ucr.edu/app/home/profile/zivran}
}

 \subjclass[2010]{14j30, 14j32, 14j60, 14j70}
\keywords{Calabi-Yau complete intersection, rigid curve, isolated curve, degeneration methods}

\begin{abstract}
	Let $X$ be either a
	general hypersurface of degree $n+1$ in $\P^n$ or a general $(2,n)$ complete intersection
	in $\P^{n+1}, n\geq 4$. We construct balanced rational curves on $X$ of all high enough degrees.
	If $n=4$ or $g=1$, 
	we construct rigid curves of genus $g$ on $X$  of all high enough  degrees.
	As an application we construct some rigid bundles on Calabi-Yau threefolds.
	In addition, we construct some low-degree balanced rational curves on hypersurfaces
	of degree $n+2$ in $\P^n$.
\end{abstract}
\maketitle
{\tt{
\section*{Declarartiions}
{\bf{Ethics approval and consent to participate}}
\par

Not applicable.\par
{\bf{Consent for publication}}\par
Consent to publish this article is granted to IJM
\par{\bf{Availability of data and materials}}
\par There are no data or materials associated with this paper.\par
{\bf{Competing interests}}\par
There are no competing interests\par
{\bf{Funding}}\par
No funding was received\par
{\bf{Authors' contributions}}\par
Ziv Ran is the sole author and contributed 100\%
\par
{\bf{Acknowledgments}}\par
I thank the Referee for numerous comments and corrections which have led 
to much improvement in the paper.
\section*{Conflict of Interest Statement}

This manuscript has not been submitted to, nor is under review at, another journal or other publishing venue.

\par	  The author has no affiliation with any organization with a direct or indirect 
financial interest in the subject matter discussed in the manuscript
\section*{Data Availability Statement}
There is no data set associated with this paper
}}
\normalsize
\section*{Introduction}

\subsection*{Setup} Let $C$ be a curve of genus $g$ on a variety $X$ of dimension $m\geq 3$
whose canonical bundle $K_X$ is ample or trivial.
Deformations of $C$ in $X$ (and sometimes with $X$) are controlled by the normal bundle $N=N_{C/X}$.
In particular, $C$ is said to be \emph{rigid}
(or sometimes \emph{isolated}) if  $H^0(N)=0$, i.e $C$ does not move on $X$, even infinitesimally. 
$C$ is \emph{regular} or \emph{strongly unobstructed} if $H^1(N)=0$
which means that deformations of $C$ in $X$ are unobstructed and not oversize.
Now in general one has
\[\chi(N)=C.(-K_X)+(m-3)(1-g).\]
This number is always $<0$ if $K_X$ is ample and $g>0$.
On the other hand, when $X$ is a Calabi-Yau ($K_X$ trivial) and either $m=3$ or
$g=1$, one always has
$\chi(N)=0$. So in that case  rigidity ($h^0=0$) is equivalent to regularity ($h^1=0$). 
One 'expects' any curve on such $X$ to be rigid: e.g. Clemens
has famously conjectured that on a general quintic threefold in $\P^4$,
all rational curves are rigid.  However the most obvious curves one can construct
are often not rigid, so it is a nontrivial problem, partially motivated by Physics 
\cite{witten1}, \cite{lin-wu-yau},
 to construct rigid, or more generally regular curves (and relatedly, vector bundles)
on Calabi-Yau manifolds, in particular those that are complete intersections in
projective space, the so-called CICY manifolds, which anyhow contain a lot of non-rigid curves.
\par
In the case $K_X$ ample, one has $C>(-K_X)<0$ in the above formula for
$\chi(N)$, hence a regular curve $C\subset X$ (more generally, a curve with $\chi(N)\geq 0$)
must anyhow have genus 0 and canonical degree $C.K_X\leq m-3$.
And whenever $0\leq\chi(N)\leq \rk(N)$ and $g=0$, 
regularity is equivalent to balancedness of $N$ (no $\O(-2)$ summands).\par
\subsection*{Known results} Results to date have largely focused
 on rigid curves on CY 3-folds. They go back to Clemens \cite{clemens-quintic}, 
who first constructed infinitely many rigid rational curves on the general
quintic in $\P^4$. This was extended by Katz \cite{katz-5tic},
then by   Ekedahl, Johnsen and Sommervoll \cite{ekedahl1}, 
to rigid rational curves on all  CICY {threefolds} (CICY3fs).
Subsequently, based on Clemens's method,
 Kley \cite{kley} constructed on  any CICY3f $X$  curves of
positive but sufficiently low genus $g$  (depending on $X$, e.g. $g<35$ for the quintic) and sufficiently
high degree (depending on the genus). For the quintic threefold, any $g\geq 0$ and
large $d$, Zahariuc \cite{zahariuc} constructs
a map of degree $d$ to its image from a smooth curve of genus $g$
to $X$ which is set-theoretically  isolated (it is not proved the map
is an embedding or  infinitesimally rigid).
Other  existence and nonexistence results for
curves of small degree (relative to the genus or otherwise) were obtained by
Knutsen \cite{knutsen}, Clemens and Kley \cite{clemens-kley}
 and Yu \cite{yu-cicy}, \cite{yu-5tic}.
I am not aware of results in the literature for curves on higher-dimensional CICYs.
Some results on vector bundles are in \cite{lin-wu-yau}, \cite{thomas-vbcy} and references therein.
On the other hand results on balanced rational and irrational curves on \emph{Fano} hypersurfaces
were obtained in \cite{caudatenormal} and \cite{elliptic}.
\par
\subsection*{New results} The purpose of this paper is to enlarge the known collection of   'good' curves
and bundles on CICY and general type  manifolds by constructing, on some $m$-dimensional CICYs,
 curves $C$ of all large enough degrees with the following
properties: \begin{itemize}\item
$C$ rigid of any genus, $m=3$; or\item
 $C$ rigid of genus 1, $m\geq 3$; or\item
 $C$ rational and balanced, $m\geq 3$; \end{itemize}
additionally, we will construct some 
balanced rational curves of degree at most $n-4$ on a general canonical
(degree- $(n+2)$) hypersurface in $\P^n$.

Precisely, we will prove the
following.
\begin{thm}\label{main-thm}
	Let $X$ be either a general hypersurface of degree $n+1$ in $\P^n$ or a general $(2,n)$
	complete intersection in $\P^{n+1}, n\geq 4$, and let $e$ be an  integer. Then\par
	(i) if $e\geq 2n-1$, there exist smooth rational curves of degree $e$
	on $X$ with normal bundle $(n-4)\O\oplus2\O(-1)$;\par
(ii) let $g\geq 1$ be an integer and assume \eqspl{e-ineq}{e\geq 2n(n-1)(g+1)+1.}
If either $g=1$ or $n=4$, then 
	there exists a smooth rigid curve of genus $g$ and degree $e$ on $X$.
	\end{thm}
As noted above,  
this yields, e.g. the first construction of rigid curves of any genus on the quintic 3-fold,
as well as the first construction of rigid or balanced curves on higher-dimensional CYs.
Note the assumption $g=1$ or $n=4$ implies $\chi(N)=0$,
 so in that case rigidity is equivalent to regularity.
For $g=0$ we have $\chi(N)=n-4$ so the curve cannot be rigid
if $n>4$. Anyhow we have $\chi(N)\geq 0$ in all cases.\par 
The case $g=1, n=4$ (already done by Kley \cite{kley}) 
has the following application to rigid vector bundles.
See \S \ref{proof-sec} for the proof.
\begin{cor}\label{bundle-cor}
	If $X$ is either a general quintic threefold in $\P^4$ or a general $(2,4)$
	complete intersection in $\P^5$, then for every  integer $e\geq 49$ $X$
	carries a rigid, indecomposable, 
	semistable rank-2 vector bundle $E$ with $c_1(E)=0, c_2(E)=e\lambda$ where $\lambda$
	is the class of a line.
	\end{cor}

On the $(2,4)$ complete intersection, two other rigid rank-2 bundles (with odd $c_1$)
were constructed by Richard Thomas \cite{thomas-vbcy}, who also constructs numerous
other rigid examples on K3 fibrations. He has also constructed in \cite{thomas-obsbundle}
an example of a curve and a  bundle on a CY3f of special moduli that are set-theoretically isolated but
not (infinitesimally) rigid. \par
Finally we will undertake a (necessarily limited) incursion into the forbidding territory
of general type by proving the following:
\begin{thm}\label{gt-thm}
	Let $X$ be a general hypersurface of degree $n+2$ in $\P^n, n\geq 5$. Then $X$ contains
	rational curves of degree $e\leq n-4$ with normal bundle 
	\[N=(n-4-e)\O\oplus (e+2) \O(-1).\]
	
\end{thm}
For $n=5, 6$ the curve in question is a line or conic and the result follows from
the elementary Lemma \ref{line-lem}.
So  Theorem \ref{gt-thm} is only interesting for $n\geq 7$. Of course here again 
$\chi(N)\geq 0$.
\par
Actually the proof of Theorem \ref{gt-thm} extends
to hypersurfaces of degree $d\geq n+2$  to yield balanced rational curves of degree 
$\leq (n-4)/(d-n-1)$.

\subsection*{Methods} The idea of the proof of Theorem \ref{main-thm} is to use a suitable degeneration of $X$
to a reducible normal-crossing variety $X_1\cup X_2$.
In the case where $X$ is a hypersurface we use a so-called quasi-cone degeneration
where $X_1$ is 'resolved quasi-cone' i.e. the blowup at a point $q$
of a quintic with multiplicity $n$ at $q$,
and $X_2$ is a degree-$n$ hypersurface in $\P^n$.
In the case where $X$ is a $(2,n)$ complete intersection, $X_0$ is a the complete intersection of a 
degree-$n$ hypersurface with
a reducible quadric, so that
 $X_1, X_2$ are both degree-$n$ hypersurfaces
in hyperplanes $H_1, H_2\simeq\P^{n+1}$ such that $X_1\cap H_1\cap H_2=X_2\cap H_1\cap H_2$.
We use results from \cite{caudatenormal} and  \cite{elliptic} about existence of curves on 
a degree-$n$ hypersurface  with
'good' normal bundle, together with a new notion, introduced in \S 1, of relative
regularity of a curve or bundle, which is analogous to
a special case of the notion of 'ultra-balance'
introduced in \cite{elliptic} but where the 'test points' are not general but lie on a divisor in
a specified system.\par
Theorem \ref{gt-thm} is proved similarly.\par
The question of existence of good curves of genus $g>0$ and high degree
 on other CICY types of any dimension, or of 
 rigid curves of genus $>1$ on any CICYs of dimension $>3$ (where $\chi(N)<0$) remains open.
 
 \par
As for general type we mention here the following
 \begin{con*}
 	A general hypersurface of degree $\geq n+2$  in $\P^n, n\geq 4$ does not contain any
 	irreducible rational curves of degree $\geq n-3$.
 	\end{con*}
The conjecture is true for $n=4$ (no rational curves at all) by \cite{voisin-e}.
As for $n=5$, an easy dimension count shows that a general septic $X\subset\P^5$ contains
 no irreducible conics. The case of higher-degree rational curves is not clear.
 \par
I would like to thank Sheldon Katz and Richard Thomas for helpful comments
and references.
\subsection*{Notation and conventions}
In this paper we work over $\C$.\par
On a curve $C$ isomorphic to $\P^1$ we denote by $\O(k)$ the line bundle of degree
$k$. A bundle on $C$ is \emph{balanced} if it has the form $a\O(k)\oplus b\O(k-1)$.
In this case the subbundle $a\O(k)$ is uniquely determined, called the 
\emph{upper subbundle}.\par
A \emph{quasi- cone} is a hypersurface $X$ of degree $d$ in $\P^n$ with a point of multiplicity
$d-1$, called the \emph{quasi-vertex}. Via projection, the blowup of $X$ in the quasi-vertex
is realized as the blowup of $\P^n$ in a $(d, d-1)$ complete intersection $Y$, where the
exceptional divisor is the birational transform of the unique hypersurface of degree ${d-1}$ containing $Y$.
The blow-up of a variety $X$ in a subvariety $Y$ is denoted by $B_YX$.

 \section{Fan, fang and quasi-cone degenerations}\label{fan-sec}

 See \cite{caudatenormal} for details.
 We recall that a \emph{fan} (also called a 2-fan) is a reducible normal-crossing variety
 of the form
 \[P_0=P_1\cup_E P_2\]
 where
 \[P_1=B_p\P^n, P_2=\P^n\]
 and $E\subset P_1$ is the exceptional divisor and $E\subset P_2$
 is a hyperplane. The family
 \[\mathcal P=B_{(p,0)}(\P^n\times\A^1)\]
 is called a standard fan degeneration and realizes $P_0$ as the
 special fibre in a family with general fibre $\P^n$.\par
 A \emph{hypersurface of type $(d_1, d_2)$} in $P_0$ has the form
 \[X_0=X_1\cup_Z X_2\]
 where \[X_1\in |d_1H_1-d_2E|_{P_1}, X_2\in |d_2H_2|_{P_2}, X_1\cap Z=X_2\cap Z\]
 	($H_1, H_2$ are the respective hyperplanes). If $d_2=d_1-1$, $X_0$ is said to
 	be of \emph{quasi-cone type}. Given a family $\bar\cX\subset\P^n\times\A^1$
 	of hypersurfaces of degree $d_1$ whose special fibre has multiplicity $d_2$
 	at $p$, its birational transform $\cX\subset\cP$ is a family of
 	hypersurfaces in $\P^n$ specializing to one of type $(d_1, d_2)$. 
 	\par
 	More generally, a \emph{fang of type} $\l$ is a variety of the form
 	\[P_0=P_1\cup_EP_2\]
 	where
 	\[P_1=B_{\P^\l}\P^n, P_2=B_{\P^{n-1-\l}}\P^n, E=\P^\l\times\P^{n-1-\l}.\]
 	This is the special fibre of the degeneration
 	\[\mathcal P=B_{\P^\l\times 0}\P^n\times\A^1\to\A^1.\]
\section{Relatively regular bundles and curves}

The purpose of this section is to study a property of vector bundles which, when applied to normal
bundles, is helpful in studying the normal bundle to a union of curves.\par
Let $C$ be a smooth curve, $L$ a line bundle on $C$,  $0\neq V\subset H^0(L)$ a linear system, 
and $D\in V$ a reduced member.
\begin{defn}
(i) A vector bundle $E$ on $C$ is said to be \emph{regular 
relative to to $D$}
if it is regular, i.e. $H^1(E)=0$ and, for any subset $D_1\subset D$
and  a general quotient  $U$ of $E|_{D_1}$ locally of rank 1,
the composite map
\[\rho_U:H^0(E)\to E|_{D_1}\to U\]
has maximal rank. 
$E$ is said to be regular relative to $V$ is it is regular relative to some divisor
(or equivalently, a general divisor) $D\in V$.\par
(ii) A curve $C$ on a variety $X$ endowed with a linear system $V$
 is said to be regular relative to $V$
if the corresponding property holds for its normal bundle $N=N_{C/X}$
and the restricted system $V|_C$.\end{defn}
More explicitly, Condition (i) means that for any distinct $p_1,...,p_k\in D$
and respective general 1-dimensional quotients $U_1,...,U_k$ of
$E_{p_1},...,E_{p_k}$, the natural map $H^0(E)\to\bigoplus\limits_{i=1}^k U_i$
has maximal rank. Or equivalently, the kernel $E'$ of the natural map $E\to \bigoplus U_i$
has either $H^1(E')=0$ or $H^0(E')=0$.\par
This notion is essentially meaningless in genus 0:
\begin{lem}
	If $D$ is any reduced divisor on $\P^1$ then any regular vector bundle is regular relative to $D$.
\end{lem}
\begin{proof} We use the above notation.
	The proof is by induction on $h^0(E)$ which  may be assumed $>0$. We may also assume
	$D_1=D$ is nontrivial and write $D=D'+p$. Let $L\subset E$ ge a line subbundle of maximal degree
	which we may assume surjects to $U|_p$ which is 1-dimensional. 
	Then letting $E'\subset E$ denote the kernel of
	$E\to U_p$, we get an exact diagram
	\[{
	\begin{matrix}
		0\to&H^0(E')&\to&H^0(E)&\to&\C_p&\to 0\\
		&\downarrow&&\downarrow&&\downarrow&\\
		0\to &U'&\to &U&\to&U|_p&\to 0
		\end{matrix}	
	}\]
As the right vertical arrow is an isomorphism and the left vertical arrow has maximal rank
by induction, it is easy to see that the middle vertical arrow likewise has maximal rank.
	\end{proof}
 In the higher-genus case relative regularity seems related to the property of
ultra-balancedness studied in \cite{elliptic}, but there is no implication either way.
Relative regularity is stronger in that the support of the quotient
is a  general divisor in $V$ which  may not be a general divisor on $C$; it is weaker in
that the quotient must have local rank 1.
\par The main result of this section is
\begin{prop}\label{regular-prop}
	For $X=\P^n, n\geq 3$ (resp. $X$ a general hypersurface of degree $n$ in $\P^n, n\geq 4$)
	and $g\geq 0$, there exists a curve of genus $g$ and degree $e\geq 2(g+1)n$ (resp. 
	$e\geq 2(g+1)n(n-1)$) in $X$ that is regular relative to $|-K_X|$.
	\end{prop}
\begin{proof}
	Case 1:  $X=\P^n$.\par
	We follow closely the proof of Theorem 29 in
	\cite{elliptic}, using induction on the genus $g$. Thus we will be
	consructing curves on a fang and smoothing them as the fang smooths to $\P^n$.
	As the case $g=0$ is automatic we begin with the case $g=1$.
	Thus let $\l$ be an integer in $[0,n-1]$ and 
	consider a fang of type $\l$ (see \S \ref{fan-sec})
	\[P_0=P_1\cup_{Q}P_2, Q:=\P^\l\times\P^{n-1-\l}\]
	where $P_1$ resp $P_2$ is the blowup of $\P^n$ in $\P^\l$ resp. $\P^{n-1-\l}$,
	with common exceptional divisor $Y=\P^\l\times\P^{n-1-\l}$. As $\P^n$ degenerates to $P_0$,
	hyperplanes have different types of limits depending on the dimension of the 
	limiting intersection with $\P^\l$.
	At one extreme, if the limiting intersection is transverse,
	 the limit in $P_0$ will have the form $H'_1\cup H'_2$
	where $H'_1$ is the birational transform of  a hyperplane in $\P^n$ transverse to $\P^\l$,
	 while $H'_2$
	is the birational transform  of a hyperplane containing $\P^{n-1-\l}$.
	This is the type we will use below.\par
	Thus let
	\[C_1\subset P_1, C_2\subset P_2\]
	be rational curves with the property that
	\[C_1\cap Q=C_2\cap Q=\{p,q\}.\] 
	Let
	\[C_0=C_1\cup_{p,q}C_2\subset P_0\]
	be a the resulting  curve of arithmetic genus 1 as in loc. cit. and set $e_i=\deg(C_i), i=1,2$ so 
	the degree of the smoothing of $C_0$ in $\P^n$ is $e=e_1+e_2-2$. As $C_0$ is a locally complete intersection
	in $P_0$, it has a locally free normal bundle denoted
	 $N_0=N_{C_0/P_0}$. Also set $ N_i=N_{C_i/\P_i}, i=1,2$. 
	 In \cite{elliptic}, Lemma 31 it is proven that we may assume $N_1, N_2$ are balanced.
	Then \[\chi(N_0)=(e_1+e_2-2)(n+1),\]
	\[\chi(N_1)=e_1(n+1)+(n-3)-2(n-\l-1), \chi(N_2)=e_2(n+1)+(n-3)-2\l.\]
	Now let $\l=[(n-1)/2]$ which makes $\chi(N_1)\leq e_1(n+1)$. Let
	\[H'_{1,1}\cup H'_{2,1},...,H'_{1,n+1}\cup H_{2,n+1}\] 
	be $n+1$ general hyperplane limits as above,	
		and let $D$ consist of $s=\chi(N_1)$ many points on $C_1\cap (H'_{1,1}\cup...\cup H'_{1,n+1})$ 
		plus $\chi(N)-s$ many
		points on $C_2\cap (H'_{2,1}\cup...\cup H'_{2, n+1})$, 
	and let $U$ be a general, locally rank $\leq 1$ quotient of $N_D$
		and let $N'$ be the kernel of $N\to U$. To prove $H^0(N)\to U$ has maximal rank
		we may assume $U$ has length $\chi(N)$, i.e. has rank exactly 1 at each point of $D$.
		Then we must prove $H^0(N')=0$.
		 Now because $N_1, N_2$
		are balanced we have $H^0(N'|_{C_1})=0$ and $H^0(N'_{C_2}(-p-q))=0$ , 
		hence $H^0(N')=0$. Thus proves that $C_0$ is regular relative to the limit 
		of $|(n+1)H|$ hence its smoothing in $\P^n$ is regular relative to $|(n+1)H|$.
		This proves our assertion in genus 1.\par
		In the general case we use induction on the genus based on a fan (i.e. fang of type
		$\l=0$) degeneration
		\[P_0=P_1\cup_E P_2, P_1=B_q\P^n, P_2=\P^n\]
		with $E\subset\P_1$ the exceptional divisor and $E\subset\P_2$ a hyperplane. We
		use a limit anticanonical divisor that is a union of limit hyperplanes of the form 
		\[D=\bigcup\limits_{i=1}^{n+1}H'_i, H'_i= H'_{1,i}\cup H'_{2,i}\] 
		with each $H'_{1,i}\subset P_1$ the
		birational transform of a hyperplane through $q$ and $H'_{2,i}\subset P_2$ a hyperplane
		with $H'_{1,i}.E=H'_{2,i}.E$. This is the opposite extreme of hyperplane limit
		types from the one used above. We consider a lci curve
		\[C_0=C_1\cup C_2\]
		where $C_2\subset P_2$ a $|-K_{P_2}|$-regular curve of genus $g-1$ and degree $e-1$
		and \[C_1= C_{1,1}\cup \bigcup\limits_{i=1}^{e-3} L_i\] 
		consists of the birational transforms of a plane cubic nodal at $q$ (so that
		$C_{1,1}.E=2$) plus lines through $q$. Then an argument similar to the above
		but simpler with $H'_{1,i}.L_j=0, \forall i,j$ shows that, for a locally-rank-1
		quotient $U$ of $N_0|_D$ the map $H^0(N_0)\to U$ is an isomorphism as required.
This concludes the proof in the case $X=\P^n$.\par
Case 2: $X\subset\P^n$ of degree $n$.\par
We use the usual 'quasi-cone' degeneration as in \cite{elliptic}, with the same notation:
\[X_0=X_1\cup_F X_2\subset P_1\cup_Q P_2.\]
Thus $P_1=B_q\P^n, P_2=\P^n$ and $Q=\P^{n-1}$ is embedded in $P_1$ as the exceptional divisor and
in $P_2$ as a hyperplane;  $X_1\subset P_1$ is the blow-up in $q$ of a general hypersurface of degree 
$n$ in $\P^n$ multiplicity $n-1$ at $q$, $X_2$ is a hypersurface of degree $n-1$ and
$F=X_1\cap Q=X_2\cap Q$.
Projection from $q$ realizes $X_1$ as the blow-up of $\P^{n-1}$ in a $(n-1,n)$ complete intersection
$Y$.	
As in loc. cit. we consider a curve
\[C_0=C_1\cup C_2\]
with $C_1\subset X_1$ the birational transform of a curve $C'_1$ of degree $e_1$ and genus
$g$ in $\P^{n-1}$ that is regular with respect to the anticanonical $|-K_{\P^{n-1}}|$,
and $C_2$ a disjoint union of lines with trivial normal bundle.
	Here as in loc. cit. $e=e_1n-a, a\leq n$. Now a limit anticanonical divisor
	on $X_0$ has the form $F_n'$ which is the birational transform of a hypersurface of degree
	$n$ in $\P^{n-1}$ containing $Y$ while $N_{C_1/X_1}$ is a general corank-1
	modification of $N_{C'_1/\P^{n-1}}$ at the $a$ points $\{p_1,...,p_a\}=C'_1\cap Y$.
		From this it is easy to see that $C_1$ is regular relative to 
		the divisor $F'_n.C_1$ and hence that
		$C_0$ is regular relative to $F'_n.C_0$. Therefore 
		as $C_0$ smooths to a curve $C$ on a general
		degree-$n$ hypersurface $X$, $C$ is regular relative to $|-K_X|$.	
	\end{proof}

\section{Lines and conics on some hypersurfaces}
In this section we will study some rational curves that will
serve as 'tails' in the construction of good curves as in Theorem \ref{main-thm}.
First, an easy remark about their normal bundle:
\begin{lem}\label{line-lem}
Let $X$ be a general hypersurface of degree $d$ in $\P^n$.\par
(i) If $d\leq 2n-3$, $X$ contains a line
	with balanced normal bundle .\par
	(ii) If $d\leq(3n-2)/2$, $X$ contains a conic with balanced normal bundle.
	\end{lem}
\begin{proof}
	Line case: Let $L$ be the line $V(x_2,...,x_n)$ in $\P^n$. A hypersurface $X$ of degree $d$ containing
	$L$ has equation $\sum\limits_{i=2}^nx_if_i(x_0, x_1), \deg(f_i)=d-1$. The normal
	sequence reads
	\[\exseq{N_{L/X}}{(n-1)\O_L(1)}{\O_L(d)}\]
		where the right map is $(f_2,...,f_n)$ which is a general map. 
		On the other hand, let $E=(n-1-d)\O_L(1)\oplus(d-1)\O_L$ be the unique balanced bundle of rank $n-2$
		and degree $n-1-d$ on $L$. Then there clearly exists a fibrewise injection
		$E\to(n-1)\O_L(1)$ whose cokernel, for degree reasons, must be $\O_L(d)$.
		By openness of the balancedness property (i.e. the fact that a balanced bundle is rigid), 
		it follows that $N\simeq E$.\par
		The conic case is similar (note that in that case the middle term in the
		normal sequence is $\O(4)\oplus (n-2)\O(2)$ which is not balanced, but
		this doesn't matter).
		
	\end{proof}
\begin{rem}\label{line-rem}
	In case $n-1\leq d\leq 2n-3$ above, the normal bundle $N=N_{L/X}$ is $N=(2n-d-3)\O\oplus(d-n+1)\O(-1)$ for the line.
	In the case $n\leq d\leq (3n-2)/2$ we have 
	 $N=(3n-2d-2)\O\oplus (2d-2n)\O(-1)$ for the conic. 
	 The argument doesn't extend to higher-degree rational curves
	 because they are not complete intersections.
	\end{rem}
/******
**************/

Let $P$ denote the blow-up of $\P^n$ at a point $q$. By a \emph{resolved quasi-cone}
in $P$ we mean the birational transform $X\subset P$ of  a quasi-cone, i.e. a
 hypersurface $\bar X$ of degree $d$ in $\P^n$
having multiplicity $d-1$ at $q$, the quasi-vertex. 
Projection from $q$ realizes $X$ as the blow-up
of $\P^n$ in a $(d, d-1)$ complete intersection curve $Y$ so that the exceptional
divisor corresponds to the unique degree-$(d-1)$ hypersurface containing $Y$, while
 hyperplane sections from $\P^n$
correspond to hypersurfaces of degree $d$ containing $Y$. In particular, taking
$d=n+1$, $(n-1)$-secant lines to $Y$ in $\P^{n-1}$ correspond to conics in $\bar X$ 
 trough the quasi-vertex  $q$.\par
 Before stating the next result we recall that for a balanced bundle
 $a\O(m+1)\oplus b\O(m), a>0$ on $\P^1$ the upper subbundle is by definition
 $a\O(m+1)$ and the upper subspace at a point $p$ is the fibre of the upper subbundle at $p$.
 \begin{lem}\label{lines-lem}
Let $X$ be either\par Case (a): a general, degree-$n$ hypersurface in $\P^n, n\geq 4$,
or\par Case (b): a general resolved quasi-cone of degree $(n+1)$  in $P$.\par
 Let $L\subset X$ 
be either\par Case (a): a general line,  or\par Case  (b): the transform of a general conic through the
quasi-vertex. \par
 Let $Z$ be either\par Case  (a):
$Z=X\cap H$ a general hyperplane section from $\P^n$, or\par Case  (b): the exceptional
divisor; and let $p=Z.L$. \par
Then:\par
(i) We have $N:=N_{L/X}=(n-3)\O\oplus\O(-1)$.\par
(ii) Varying $X$ with fixed $L$ and $Z$, 
and identifying $N|_p\simeq T_pZ$, the upper
subspace $(n-3)\O|_p\subset T_pZ$ becomes a general hyperplane.
\par (iii) There is a deformation $\{(X_t, L_t):t\in T\}$, fixing $Z$, such that 
the image of the map
\[T\ni t\mapsto L_t\cap Z\]
contains a neighborhood of $p$. 	
	\end{lem}
\begin{rem}
 From the proof below
it will follow that $X$ actually contains a 1-parameter family of such curves.
	\end{rem}
\begin{proof}[Proof of Lemma] Note that all our assertions are open in $L$  $Z$ and $X$,
so it suffices to prove them for some special case.\par
Case (a):\par
	(i) This is just Lemma \ref{line-lem} above. Note that it implies that $L$ moves on $X$
	in a smooth $(n-3)$-dimensional family and, because the restriction map
	$H^0(N)\to N|_p$ for general $p\in L$ has $(n-3)$-dimensional image, this family fills
	up an $(n-2)$-dimensional scroll that is a divisor on $X$.\par
	(ii) Let $x_0$ be the equation of $H$ and $x_1,...,x_n$ be general linear forms, and
	consider the case of a simplex
	\[X_0=H_1\cup...\cup H_n, H_i=V(x_i)\]
	and let $A_1\cup...\cup A_n=X_0\cap H$ be its $H$-section.
	$X_0$ has singular locus
	\[S=\bigcup S_{ij}, S_{ij}:=H_i\cap H_j, 1\leq i<j\leq n.\]
	Let $L\subset H_1$ be a general line and 
	\[p_{j}=L. H_j=L. S_{1j}, j=2,...,n; p=L. H=L.A_1.\]
Set $f_0=x_1\cdots x_n$, let $g$ be a general degree-$(n-1)$ form vanishing at $p_2,..., p_n$,
	and let $X_1\subset\P^n$ be the hypersurface with equation $f_0+x_0g$. Thus $X_1$ contains
	$p_2, ..., p_n$ and because $x_0$ is the equation of $H$ in $\P^n$, one has
	$X_1\cap H=X_0\cap H=Z$ and this has equation $f_0$ in $H$.
			Set $Q_j=V(g)\cap S_{1j}, j=2,...,n$, which is a degree-$(n-1)$ subvariety in $S_{1j}$
			containing $p_j$,
			and $X_1\cap S_{1j}$ consists of the hyperplane $V(x_0)$ plus $Q_j$.
			Let $t$ be a coordinate on $\A^1$ and
		consider the linear family, depending on $g$
		\[\pi=\pi(g): \mathcal X(g)=V(f_0+tx_0g)\subset\P^n\times\A^1\to\A^1, X_t=\pi\inv(t).\]
		This is a pencil of hypersurfaces with fixed $H$-section $Z$. 
		Then $\mathcal X(g)$ is singular at $S\cap X_1$ and, 
		away from  $S_0=S\cap V(x_0, g)$,
		$\mathcal X_g$ has singularity of type (3-fold ordinary double point)$\times \A^{n-3}$ and so admits
		a small resolution $\mathcal X'\to\A^1$ with fibre $X'_0$. Note
		that $S_0$ also coincides with the singular locus of a general fibre $X_t$. \par
		The normal
		bundle $N_{L/X_0'}$ is a corank-1 down modification (i.e. subsheaf of colength 1)
		of $N_{L/H_1}$ at
		$p_{2}, ..., p_{n}$. Identifying $N_{L/H_1}|_{p_{j}}\simeq T_{p_{j}}S_{1j}$,
		this is the down modification corresponding to the subspace $T_{p_{j}}Q_j, j=2,..., n$.
		Since these subspaces may be chosen generally it follows firstly  that the modification is general,
		so that $N_{L/X'_0}=(n-3)\O\oplus\O(-1)$. Moreover clearly, and as one can check by a coordinate
		computation, as the hyperplanes  $T_{p_{j}}Q_j, j=2,..., n$ vary, so does the 
		$(n-3)\O$ subsheaf 
		and its fibre at $p$. Explicitly, write $N_{L/H_1}=L_2\oplus...\oplus L_{n-1}$ where $L_i\simeq \O(1)$
		and the fibre of $L_i$ at $p_i$ corresponds to $T_{p_i}Q_i, i=2,...,n-1$. 
		Then the down modification corresponding to $T_{p_i}Q_i, i=2,...,n-1$ replaces each $L_i$
		by $\O_L$ so it is just $(n-2)\O_L$ with basis $e_2,...,e_{n-1}$.
		Then if the hyperplane $T_{p_n}Q_n$ is represented by $(\alpha_2,...\alpha_{n-1})$ in
		this basis, then the $(n-3)\O$ subsheaf is generated
		by $\alpha_3e_2-\alpha_2e_3,...,\alpha_{n-1}e_{n-2}-\alpha_{n-2}e_{n-1}$ 
		and this clearly moves with $T_{p_n}Q_n$.
		\par
		(iii) What (ii) shows is that the set of limit lines in the family $X(g)$ is
		a smooth $(n-3)$-parameter family which traces out on $A_1$ a smooth divisor $D(g)$ whose tangent 
		hyperplane at $p$
		corresponds to the aforementioned $(n-3)\O$ subsheaf. This depends on $g$ through the $Q_i$ curves.
		As $g$ varies, the latter computation shows that $D(g)$ will vary and with it the tangent
		hyperplane $T_{p}D(g)$. Therefore the divisors $D(g)$ will sweep out $A_1$, filling up an open set.\par
		Case (b):\par
		It is convenient to
		start with a line $L\subset\P^{n-1}$ with $n-1$ distinct points $p_1,...,p_{n-1}$,
		then let $F_{n}, F_{n+1}$ be general hypersurfaces of the indicated degrees containing
		$p_1,...,p_{n-1}$, then let $Y=F_n\cap F_{n+1}$ and $P=B_Y\P^{n-1}$. 
		Then $N_{L/P}$ is the down modification of $N_{L/\P^{n-1}}=(n-2)\O(1)$ in the rank-1 quotients
			corresponding to $T_{p_i}Y, i=1,...,n-1$.
		Because these tangent spaces are general, we have $N_{L/P}=(n-3)\O\oplus\O(-1)$.
		Then because we can vary $Y$ while fixing $Z=F_{n-1}$,
		  assertions (ii-iii) are clear.
	\end{proof}
\section{Conclusions}\label{proof-sec}
\begin{proof}[Proof of Theorem \ref{main-thm}] 
\hskip1cm

 \ul{$(2,n)$ complete intersection case:}\par
 Assume first $e=2e_1$ even. We first prove assertion (ii), so we assume $g=1$ or $n=4$. 
Let \[X_0=X_1\cup X_2\subset \P^{n+1}\]
where  $X_q, X_2$ are general degree-$n$ hypersurfaces in respective hyperplanes $P_1, P_2\subset\P^{n+1}$ 
with the property  that
\[X_1\cap P_1\cap P_2=X_2\cap P_1\cap P_2=:Z.\] 
We may assume $Z$ is a general hyperplane section of $X_1$ and $X_2$.
A smoothing of $X_0$ is given by a smoothing of the reducible quadric $P_1\cup P_2$,
and has total space that is singular along a divisor $Z_q\subset Z$.
\par
Let $C_1\subset X_1$ be a curve of degree $e_1=e/2$ and genus $g$, regular
with respect to $|\O_{X_1}(1)|=|-K_{X_1}|$ (cf. Proposition \ref{regular-prop}), and
meeting $Z$ transversely in $p_1,...,p_{e_1}$ and disjoint from $Z_q$. 
As $(g-1)(n-3)=0$ and $N_{C_1/X_1}$ is balanced, 
$C_1$ moves in a smooth family of dimension  $e_1+n-5=\chi(N_{C_1/X_1})$
on $X_1$. 
Now because
\[e_1-2=e_1+n-4-(n-2)=\chi(N_{C_1/X_1}(-p_i))>0,\] hence $H^1(N_{C_1/X_1}(-p_i)=0$, so each $p_i$ moves on $Z$
filling up an analytic open set $U_i$. By restricting, we may assume the $U_i$
pairwise disjoint. On the other hand consider a balanced line $L\subset X_2$ with normal bundle
$(n-3)\O\oplus\O(-1)$.
As we saw in Lemma \ref{lines-lem}, as the pair $(X_2, L)$ moves while fixing $Z$,
 the point $L\cap Z$ moves, filling up a Zariski open set $V$ (NB it is obviously
 necessary here that the hypersurface move together with the line).
As $V$ must intersect each $U_i$, we may assume that we have a balanced line line $L_i$ in $X_2$
through $p_i$ for $i=1,...,e_1$. Moreover by Lemma \ref{lines-lem}, Case (a),
 the may assume the upper subspace $M_i$ of $N_{L_i/X_2}|_{p_i}$
is general as subspace of $T_{p_i}Z$. Let $N_1\subset N_{C_1/X_1}$ be the down modification
corresponding to $M_1,...,M_{e_1}$ and $N_0=N_{C_0/X_0}$. \par
Now recall we are assuming either $g=1$ or $n=4$.
Then because $N_{C_1/X_1}$ is regular relative to $|\O(1)|$ and the modification
is general, we have $H^0(N_{1})=0$.  let
\[C_2=\bigcup\limits_{i=1}^{e_1} L_i, C_0=C_1\cup C_2.\]
Now note that $H^0(N_{C_0/X_0})=H^0(N_1)$ and as we have seen this vanishes. 
Thanks to our assumption that either $n=4$ or $g=1$, we have $\chi(N_{C_0/X_0})=0$
so it follows that $H^1(N_{C_0/X_0})=0$ as well.
Thus, $C_0$ is rigid on $X_0$
and deforms with it to a rigid curve of  degree $2e_1=e$ on a general 
$(2,n)$ complete intersection in $\P^n$. This completes the proof of assertion (ii) $e$ is even.
\par
As for Assertion (i), i.e. the case $g=0$, still assuming $e$ is even,
the argument is  is similar but simpler. Using the same notation, note that it suffices to prove $h^0(N_0)=n-4$. 
We have an exact sequence
\[0\to H^0(N_0)\to H^0(N_{C_1/X_1})\oplus H^0(N_{C_2/X_2})\stackrel{\rho}{\to}\bigoplus\limits_{i=1}^{e_1} M_i\]
By Proposition \ref{regular-prop}, we may assume $C_1\subset X_1$ is regular with respect to $\O(1)$. Then because
the $M_i$ are general  hyperplanes in $N_{C_1/X_1}|_{p_i}$ it follows that 
\[h^0(N_0)=h^0(N_{C_1/X_1})-e_1=n-4.\]

 ******
************* \par
Now consider the case $e=2e_1+1$ odd. The idea is to use the same $C_1$ of degree $e_1$ and to
 replace one of the lines, say $L_1$,
by a suitable conic $M$. Now recall that the smoothing of $X_0$ corresponds to smoothing
the reducible quadric with equation $x_1x_2$ to one with equation $q$. The total space of
the family has local equation $x_1x_2+tq$. As such it is singular in $x_1=x_2=t=q=0$,
i.e. the intersection $Z_q$ of $Z$, the double locus of the special fibre,  
with the quadric $q=0$. There the total space admits
a small resolution $\tilde{\mathcal X}$
 by blowing up $x_1=q_0$ (this makes sense globally) 
 and the special fibre in    $\tilde{\mathcal X}$ replaces the component $X_2$
by its blowup in $Z_q$. Choosing the quadratic  $q$ suitably, we can arrange that $M$ meets $Z_q$
in exactly 1 point and there transversely.
Then the birational transform of $M$ meets that of $Z$ in exactly 1 point and has
normal bundle $(n-3)\O\oplus\O(-1)$ by Lemma \ref{line-lem}, so we can proceed as before, for both
assertion (i) and (ii).

\ul{Hypersurface Case :}\par
Suppose $e$ is even. Here we use a standard quasi-cone degeneration to
\[X_0=X_1\cup_Z X_2\]
with $X_1$ a resolved  quasi-cone of degree $n+1$ and $X_2$ 
is a hypersurface of degree $n$ in $\P^n$. The degeneration has smooth total space. 
As in the $(2,n)$ case, we have
a lci curve
\[C_0=C_1\cup C_2\]
with $C_1$ a disjoint union of $e_1=e/2$ many 'conics' with normal bundle $(n-3)\O\oplus\O(-1)$
as in Lemma \ref{lines-lem}, Case (b), and $C_2$ is a curve of genus $g$ and degree $e_1=e/2$
on $X_2$ that is regular relative to $|-K_{X_2}|=|\O_{X_2}(1)|$, and with
\[C_1\cap Z=C_2\cap Z.\] Then an argument as above
shows that $H^0(N_{C_0/X_0})=0$ so we can conclude as above. \par
Finally in case $e$ is odd we replace one of the 'conics' by a 'twisted cubic'. This
is obtained by starting with a conic $M\subset\P^{n-1}$ with $2n-1$
points $p_1,...,p_{2n-1}$, choosing general hypersurfaces $F_n, F_{n+1}$ through $p_1,...,p_{2n-1}$
and blowing up $Y=F_n\cap F_{n+1}$. The birational transform $\tilde M$ of $M$
 meets $Z=\tilde F_n$ in 1 point and contributes 2 to the total degree of $C_0$ and its smoothing.
 Even though the $M$ has normal bundle $N=\O(4)\oplus(n-3)\O(2)$
 which is unbalanced, the down modification of $N$ in $>0$ points is balanced
and so $\tilde M$   has balanced normal bundle i.e. $(n-2)\O\oplus\O(-1)$.
\end{proof}
\begin{rem}
	It is proved in \cite{voisin-e} that the general sextic in $\P^4$ contains no rational or elliptic curves,
	thus showing that Theorem \ref{gt-thm} is sharp.
	\end{rem}
\begin{rem}\label{filling}
	Let $C$ be a rational curve on a CICY $X$ as in Theorem \ref{main-thm}
	with normal bundle $N=(n-4)\O_C\oplus2\O_C(-1)$.  Then for $p\in C$ the image of restriction
	$H^0(N)\to N|_p$ is $(n-4)$-dimensional. It follows that 
	 $C$ moves in $X$ filling up an $(n-3)$-dimensional ruled subvariety birational to a fibration
	 with fibre $\P^1$, i.e. either a $\P^1$-bundle ($e$ odd) or a conic bundle ($e$ even).
\par	 For example, for $n=5$ and $X\subset\P^5$ a sextic, such a rational curve fills up
	 a surface. When $C$ is a line, the degree of the surface is easily computed
	 to be $d=c_7(\sym^6(Q)).c_1(Q)$ where $Q$ is the tautological rank-2 bundle
	 on the Grassmannian $\gG(1, 5)$. Evaluating this number
	 is a routine
	 if tedious calculation (which we have undertaken at the referee's behest). 
	 If $r_1, r_2$ are the Chern roots of $Q$ then those of $\Sym^6(Q)$ are
	 $6r_1, 5r_1+r_1,...,6r_2$, hence
	 \[c_7(Q)=4320r_1r_2(r_1^5+r_2^5)+37584 r_1^2r_2^2(r_1^3+r_2^3)+98064 r_1^3r_2^3(r_1+r_2)\]
	 hence, where $c_i=c_i(Q)$,
	 \[d=4320c_1^6c_2+15984c_1^4+6912c_1^2c_2^3\]
	 The Chern numbers is question are, respectively,  $8!144=5806080, 2, 1$
	 hence
	 \[d=25082304480.\]
	\end{rem}
\begin{rem}
	If $C$ is as in Remark \ref{filling} then clearly $T_X|_C=\O(2)\oplus(n-2)\O\oplus\O(-1)$
	which is not balanced. Thus in the terminology of \cite{elliptic}, $C$ is never
	ambient-balanced.
	\end{rem}
\begin{rem}
	The above method of constructing curves yields rigid nodal lci curves of any genus  on $X_0$
	for any $n$. However for $n>4, g>1$ these curves have $H^1(N)\neq 0$, so it's not clear these
curves	smooth out with $X_0$.
	 
	\end{rem}
\begin{proof}[Proof of Corollary \ref{bundle-cor}] 
	The proof is based on the Serre construction
	(see \cite{okonek}, \S I.5.1 which, though formulated for projective spaces is actually
	mostly valid for arbitrary smooth varieties, as already noted in \cite{okonek}, \S I.5.3).
	If $C$ is an elliptic curve as in  Theorem 1, Part (ii),
	then the computations in \cite{okonek}, \S I.5.1 show that $\mathcal{E}xt^1(\I_C, \O_X)=\O_C$
	and there is an exact sequence
	\[H^1(\O_X)\to\mathrm{Ext}^1(\I_C, \O_X)\to H^0(\mathcal{E}xt^1(\I_C, \O_X))\to H^2(\O_X).\]
	Since the extreme groups clearly vanish, 
	the element $1\in H^0(\O_C)$ yields 
	a uniquely determined sheaf $E$ as an extension of $\I_C$ by $\O_X$, i.e. one has exact
	\eqspl{bundle}{ 
	\exseq{\O_X}{E}{\I_C}.
	}
	Then the computations
	in loc. cit. show that $E$ is locally free with $c_1(E)=0, c_2(E)=[C]=e\lambda$ (the inequality $e\geq 49$
	 comes from \eqref{e-ineq}).\par
	Now an easy diagram chase around the exact sequence
	\[\exseq{\O_X}{E}{\I_C}\]
	shows that $h^0(E)=1, h^1(E)=0$, so the unique section up to scalars of $E$ extends
	to (infinitesimal) deformations, therefore
	$C$ deforms with deformations of $E$. 
	Conversely the functoriality of the Serre construction shows that deformations of $C$
	induce deformations of $E$. Thus, there is an isomorphism between the deformation
	functors of $C$ and $E$.
	Therefore since $C$ is rigid, so is $E$. Finally since $Pic(X)$ is generated
	by the hyperplane class,
	indecomposability follows from the Chern classes while (proper) semistability follows
	from the above exact sequence.
\end{proof}
%
\begin{proof}[Proof of Theorem \ref{gt-thm}]
	For the proof we use a standard quasi-cone degeneration where $C\subset X\subset\P^n$ degenerates to 
	\[C_0=C_1\subset X_1\subset X_0=X_1\cup_Z X_2\subset P_0=P_1\cup P_2\]
	where: $X_1\subset P_1$ a resolved quasi-cone of degree $n+2$, i.e. 
	the birational transform of a hypersurface of degree $n+2$ with a point of multiplicity $n+1$,
	also realizable as  $B_Y\P^{n-1}$, the blowup of 
	$\P^{n-1}$ in $Y$ which is a codimension-2 complete intersection
	of general hypersurfaces  $F_{n+1}$,  $F_{n+2}$ of the indicated degrees;
	$X_2\subset P_2=\P^n$ a degree-$(n+1)$ hypersurface;   $Z=X_1\cap X_2\simeq F_{n+1}$ 
	coincides with the birational
	transform of $F_{n+1}$.
	For $C_1\subset X_1$ we take a curve constructed as follows. Start with a general rational curve 
	$C'_1$ of degree $e$ in $\P^{n-1}$, i.e. a rational  normal curve in a general $\P^e\subset\P^{n-1}$,
	let $F_{n+1}$ be a general hypersurface and then let $F_{n+2}$ be a general
	hypersurface through $C'_1\cap F_{n+1}$. Then take for $C_1$ the birational transform
	of $C'_1$ in the blowup of $Y=F_{n+1}\cap F_{n+2}$.  
	Because $C'_1\cap Y=C'_1\cap F_{n+1}$, we have $C_1\cap Z=\emptyset$. 
	Because $N_{C'_1/\P^e}$ is balanced, i.e. $N_{C'_1/\P^e}=(e-1)\O(e+2)$,
	it follows that
	\[N_{C'_1/\P^{n-1}}=(e-1)\O(e+2)\oplus(n-e)\O(e).\]
	Then because the tangent spaces to $Y$ at $Y\cap C'_1$ can be specified
	arbitrarily (compare \cite{caudatenormal}, proof of Theorem 20), 
	it follows that $N_{C_1/X_1}$ is a general, locally corank-1 down modification
	of $N_{C'_1/\P^{n-1}}$ at $e(n+1)$ points, it follows that
	\[N_{C_1/X_1}=(n-e-4)\O\oplus(e+2)\O(-1).\]
	Because $H^1(N_{C_1/X_1})=H^1(N_{C_1/X_0})=0$ as $C_1$ is disjoint from $Z$,
	it follows that $C_1$ deforms along as $X_0$ smooths to a genral degree-$(n+2)$
	hypersurface in $\P^n$.
	\end{proof}
	\begin{rem}If $X_d\subset\P^n$ is a hypersurface of degree $d$ and $C\to X_d$
	is a curve of degree $e$ and genus $g$ with normal bundle $N$, then 
	$\chi(N)=(n+1-d)e+n-4-g(n-2)$. Thus if $d>n+1$ and $e>\frac{n-4}{d-n-1}$
	(any $g\geq 0$), then $\chi(N)<0$
	and in particular $H^1(N)\neq 0$. Indeed one does not expect such a  curve (e.g. rational curve)
	to exist on a \emph{general} $X_d$. In \cite{relreg} however one can find some constructions
	for 'well-behaved' families of \emph{special} hypersurfaces $X_d$,
	codimension $h^1(N)=-\chi(N)$
		in the family of all hypersurfaces,  endowed with 'good' curves.
\par
	Notations as above, the above argument in the proof of Theorem  \ref{gt-thm},
	applied to a rational curve $C_0$ of degree $e>n-4$, yields a curve
	$C_1$ whose normal bundle $N_{C_1/X_1}$ has $\O(-2)$ summands, hence has
	$H^1\neq 0$, so it's not clear if $C_1$ deforms along as $X_0$ deforms to a general $X_d$.
	\par
On the other hand whenever $d\geq n+2$ the same argument, using a degree-$d$ quasi-cone degeneration,
 does produce balanced rational curves
	of degree $e\leq (n-4)/(d-n-1)$. In view of Lemma \ref{line-lem}, 
	this is interesting when we can take $e\geq 3$ which means  $3\leq (n-4)/(d-n-1)$ i.e.
	$n+2\leq d\leq (4n-1)/3$, $n\geq 7$. For example, taking $n=10, d=13$, we get a balanced twisted cubic on
	a general degree-13 hypersurface 
	$X_{13}\subset\P^{10}$.
	\end{rem}

\bibliographystyle{amsplain}
\bibliography{../mybib}
\end{document}

%% file: myheader3.tex
\def\inv{^{-1}}

\def\gG{\mathbb G}

\def\refp #1.{(\ref{#1})}

\newcommand{\A}{\mathcal{A}}

\newcommand{\ul}[1]{\underline {#1}}

\def\sbr #1.{^{[#1]}}
\def\sfl #1.{^{\lfloor #1\rfloor}}

\def\inv{^{-1}}
\def\?{{\bf{??}}}

\def\A{\Bbb A}

\def\C{\mathbb C}
\def\P{\mathbb P}

\def\sym{\text{\rm Sym} }

\def\O{\mathcal O}

\def\Sym{\textrm{Sym}}

\def\rk{\text{rk}}

\def\g{\mathfrak g}

\def\1/2{\frac{1}{2}}

\def\I{\mathcal{ I}}

\def\2{{[2]}}
\def\l{\ell}
\def\nl{\newline}

\def\<{\langle}
\def\>{\rangle}

\def\2{{[2]}}
\def\l{\ell}

\def\scl #1.{^{\lceil#1\rceil}}
\def\spr #1.{^{(#1)}}
\def\sbc #1.{^{\{#1\}}}

\def\subpr#1.{_{(#1)}}

\def\beq{\begin{equation*}}
\def\eeq{\end{equation*}}

\def\g3{{\Gamma\spr 3.}}

\newcommand{\eqspl}[2]{
\begin{equation}\label{#1}
\begin{split}
#2\end{split}\end{equation}}

\newcommand{\exseq}[3]{
0\to #1\to #2\to #3\to 0
}

\newcommand{\beginalphaenum}{
\begin{enumerate}\renewcommand{\labelenumi}{ }
\item \begin{enumerate}
}

\def\eex{\end{rm}\end{example}}